%% file: main.tex
\documentclass[journal,twoside,web]{ieeecolor}
\usepackage{lcsys}
\usepackage{cite}
\usepackage{amsmath,amssymb,amsfonts}
\usepackage{algorithmic}
\usepackage{graphicx}
\usepackage{textcomp}
\usepackage{comment}

\usepackage[colorlinks = true,
            linkcolor = cyan,
            urlcolor  = blue,
            citecolor = blue,
            anchorcolor = blue
            ]{hyperref}

\newtheorem{theorem}{Theorem}

\newtheorem{proposition}{Proposition}

\newtheorem{lemma}{Lemma}

\newtheorem{assumption}{Assumption}

\newtheorem{example}{Example}
\newtheorem{problem}{Problem}
\newtheorem{remark}{Remark}

\usepackage[nameinlink]{cleveref}
\crefname{equation}{}{}
\crefname{theorem}{Theorem}{Theorems}
\crefname{corollary}{Corollary}{Corollaries}
\crefname{example}{Example}{Examples}
\crefname{problem}{Problem}{Problems}
\crefname{assumption}{Assumption}{Assumptions}
\crefname{lemma}{Lemma}{Lemmas}
\crefname{proposition}{Proposition}{Propositions}
\crefname{figure}{Figure}{Figures}
\crefname{table}{Table}{Tables}
\crefname{fact}{Fact}{Facts}
\crefname{conjecture}{Conjecture}{Conjectures}
\crefname{section}{Section}{Sections}
\crefname{appendix}{Appendix}{Appendices}
\Crefname{equation}{}{}
\Crefname{theorem}{Theorem}{Theorems}
\Crefname{corollary}{Corollary}{Corollaries}
\Crefname{example}{Example}{Examples}
\Crefname{lemma}{Lemma}{Lemma}
\Crefname{proposition}{Proposition}{Proposition}
\Crefname{figure}{Figure}{Figures}
\Crefname{table}{Table}{Tables}
\Crefname{section}{Section}{Sections}
\Crefname{appendix}{Appendix}{Appendices}

\input{commands.tex}

\begin{document}

\title{On the Strong Duality in Continuous-time and Discrete-time Linear Quadratic Regulators}

\author{Yuto Watanabe, \IEEEmembership{Student Member, IEEE}, and Yang Zheng, \IEEEmembership{Senior Member, IEEE}
\thanks{This work is supported by NSF CMMI 2320697 and NSF CAREER 2340713. Y. Watanabe and Y. Zheng are with the Department of Electrical
and Computer Engineering, University of California San Diego; \texttt{\{y1watanabe,zhengy\}@ucsd.edu}.}}

\maketitle
\thispagestyle{plain}
\vspace{-6mm}

\begin{abstract}
This paper revisits the strong duality in the linear quadratic regulator (LQR)
for continuous-time and discrete-time systems, and explores its interconnection with typical assumptions and the uniqueness of primal-dual solutions.
Using a linear operator $\Psi$, we formulate a common nonconvex LQR problem that captures both time domains. We then derive its Lagrange dual problem and establish the strong duality via
a rank-constrained tight semidefinite program (SDP) relaxation.
Further, we show that the primal-dual optimal solutions to the SDP relaxation, after dropping the rank constraint, recover the classical~algebraic Riccati equations and optimal feedback gains in a constructive manner.
The dual derivation and strong duality analysis rely on mild standard assumptions and exploit the properties of the linear operator and its adjoint,
revealing a structural symmetry between the two time~domains.
\end{abstract}


\input{arXiv-sections/section_i}
\input{arXiv-sections/section_ii}

\input{arXiv-sections/section_iii}
\input{arXiv-sections/section_iv}

\section{Conclusion}\label{section:conclusion}

We presented a unified duality analysis for~continuous-time and discrete-time LQRs.
Using the linear operator $\Psi$, we formulated a common nonconvex primal problem for both time domains.
Deriving the Lagrange dual,
we demonstrated the strong duality and optimal solutions that recovered the LQR optimal gains and AREs. 
Moreover, we showed that the common formulation admits a tight SDP relaxation with the same dual problem, highlighting an alternative aspect of LQR as a well-behaved QCQP.

\bibliographystyle{IEEEtran}
\bibliography{ref.bib}

\appendix

\crefalias{section}{appendix}
\crefalias{subsection}{appendix}
\crefalias{subsubsection}{appendix}

\input{arXiv-sections/appendix}
\end{document}

%% file: commands.tex
\newcommand{\tr}{{{\mathsf T}}}
\newcommand{\Tr}{{{\mathbf{Tr}}}}
\newcommand{\her}{{{\mathsf H}}}

\newcommand{\sdp}{\texttt{sdp}}

\newcommand{\ct}{\texttt{ct}}%
\newcommand{\dt}{\texttt{dt}}%

\newcommand{\inner}[2]{\left\langle #1,#2\right\rangle}

%% file: arXiv-sections/section_i.tex
\section{Introduction}
\label{sec:introduction}

The Linear Quadratic Regulator (LQR) is arguably the most fundamental optimal control problem \cite{kalman1960contributions},
with significant theoretical importance and practical applications.
The goal is to design an optimal controller for a linear dynamical~system by minimizing a quadratic cost accumulated along the system trajectory.
The LQR problem is known to enjoy many favorable theoretical features.
For example, an elegant fact is that in the infinite-horizon formulation, the optimal control policy is linear and static of the form 
$u=Kx$.
It also~admits a convex semidefinite program (SDP) reformulation through a change of variables. Moreover, a nonconvex LQR formulation in the policy space exhibits a strongly convex-like landscape, known as \textit{gradient dominance}~\cite{fazel2018global,watanabe2025revisiting,watanabe2026gradient,mohammadi2019global}.
Recently, 
as an alternative perspective, 
Bamieh \cite{bamieh2024linear} and our recent work \cite{watanabe2025revisiting} revisited LQR using an SDP and duality strategy \cite{vandenberghe1996semidefinite},
formalizing a duality-based derivation of the optimal gain.
Our work \cite{watanabe2025revisiting}~also shows that for continuous-time systems, a nonconvex LQR can be viewed as a quadratically constrained quadratic programming (QCQP) with a tight SDP relaxation and zero duality gap.

It is known that both continuous-time and discrete-time LQR problems 
satisfy the above favorable properties almost in the same way.
However, the majority of existing studies
have dealt with 
each time domain separately
(including the duality analyses \cite{vandenberghe1996semidefinite,bamieh2024linear,watanabe2025revisiting}).
While such a treatment is natural
 and simplifies the exposition, this makes the symmetry and subtlety between the two time domains less explicit.
Toward closing this gap, this paper revisits the strong duality and develops a common framework for continuous-time and discrete-time LQRs.
The topic of strong duality in LQR is not new, and there exist many duality-based analyses and synthesis for linear control \cite{yao2001primal,balakrishnan2003semidefinite,scherer2006lmi,rantzer1996kalman,gattami2009generalized,you2016direct,watanabe2025semidefinite}.  
To our knowledge, continuous-time and discrete-time LQRs have been analyzed separately via different proof techniques \cite{lee2018primal,bamieh2024linear,watanabe2025revisiting,zhao2023global,yao2001primal}.
In these works, slightly different assumptions were made, particularly in terms of the stochasticity in initial conditions and the uniqueness of the optimal gain.
The work \cite{zhao2023global} showed strong duality for a discrete-time LQR with a risk constraint, but it requires stronger assumptions on the weight and covariance matrices.
In contrast, we impose only basic and standard assumptions, which cover LQRs with both stochastic and deterministic initial conditions. Further, we  
work with a common LQR formulation for two time domains using
   a linear operator $\Psi$ \cite{watanabe2025revisiting},
   inspired by earlier works for linear matrix inequalities (LMIs) \cite{iwasaki2005generalized,scherer2006lmi,you2016direct}.
   While this formulation is nonconvex, we can
   derive the Lagrange dual 
   and an SDP relaxation,
   similarly to Shor's relaxation of QCQPs \cite{shor1987quadratic}.
   We then demonstrate the strong duality by showing the tightness of this relaxation and by utilizing the SDP duality theory.

   Our results provide the following insights. 
   First, our results clarify that 
   in the strong duality analysis, algebraic Riccati equations (AREs), and optimal LQR gains,
   the distinction~between the two time domains for LQR is purely encoded in the linear operator $\Psi$.
    In our proof, similar to Shor's SDP relaxation for QCQPs,
    we derive a common SDP relaxation from a rank-constrained form of the LQR.
    The SDP relaxation
    turns out to be lossless, and the desirable strong duality follows from 
    the analysis of its Karush-Kuhn-Tucker (KKT) conditions.
     Second,
   we clarify that strong duality holds for both time domains under mild standard assumptions, covering both deterministic and stochastic initial conditions.
   Compared with \cite{yao2001primal,lee2018primal,zhao2023global}, our result only requires the initial covariance matrix $\mathbb{E}[x_0x_0^\tr]$ to be positive semidefinite (\cref{theorem:strong_duality}). This condition includes the deterministic LQR as a special case,~is~sufficient for strong duality,
   and separates strong duality from the uniqueness of the optimal gain.
   This also indicates that strong duality is a less fragile manifestation of benign nonconvexity than the gradient dominance \cite{fazel2018global,mohammadi2019global,watanabe2025revisiting,watanabe2026gradient}.
    Finally,
    our rank-structure-based proof identifies the optimal gains \textit{constructively}, without assuming their form in advance. 
    This may be viewed as a benefit of our approach compared with \cite{lee2018primal,bamieh2024linear}, and further enables us to clarify when the optimal LQR gain is unique.
    In particular, we clarify that the primal SDP relaxation always admits a unique solution (\cref{lemma:strong-duality-SDP}), while
    the primal nonconvex problem and the dual may have many optimal solutions when $\mathbb{E}[x_0x_0^\tr]\nsucc 0$ (\cref{proposition:optimal-solutions}).
   
The rest of this paper is organized as follows.
\cref{section:problem-formulation} formulates the LQR problems and presents the problem statement.
Then, we state our main result in \cref{section:strong-duality}, and present its proof in \cref{section:proof_duality}.
Finally, \cref{section:conclusion} concludes the paper.

%% file: arXiv-sections/section_ii.tex
\section{Problem formulation}\label{section:problem-formulation}
Here, we formulate the LQR problems for continuous-time and discrete-time systems,
and provide our problem statement.

\subsection{Continuous-time and discrete-time LQR problems}
\subsubsection{Continuous-time LQR} Consider a continuous-time linear time-invariant (LTI) system
\begin{equation}\label{eq:dynamic_continuous}
    \dot{x}(t)=A x(t)+B u(t),
\end{equation}
where $x(t)\in\mathbb{R}^n$ is the state and $u(t)\in\mathbb{R}^m$ is the input. Let $x(0)=x_0\in\mathbb{R}^n$, and assume that $x_0$ is a random variable with zero mean and covariance
$
W=\mathbb{E}[x_0x_0^\tr]\in\mathbb{S}_+^n
$, where $\mathbb{S}_+^n$ is the set of $n\times n
$ positive semidefinite matrices.

Given weight matrices $Q\in\mathbb{S}_+^n$ and $R\in\mathbb{S}_{++}^m$ (denoting the set of $m\times m$ positive definite matrices),  consider the infinite-horizon LQR problem
to minimize the cost $J_{\ct}(K)$ below:
\begin{equation}\label{eq:LQR-stochastic-noise_continuous}
\begin{aligned}
\min_{K}\;
J_{\ct}(K)\!:=\!
\mathbb{E}\!\left[
\int_0^\infty \bigl(x(t)^\tr Qx(t)+u(t)^\tr Ru(t)\bigr)\,dt
\right]
\end{aligned}
\end{equation}
over all static state-feedback controllers of the form  $u(t)=Kx(t),$
where $K\in\mathbb{R}^{m\times n}$ belongs to the set of stabilizing gains
$
\mathcal{K}_{\ct}
\!:=\!
\left\{
K
\;\middle|\;
\operatorname{Re}\bigl(\lambda_i(A+BK)\bigr)<0,\, \forall i=1,\dots,n
\right\}
$.
It is~well known that this class contains the continuous-time LQR optimal controller. Hence, \cref{eq:LQR-stochastic-noise_continuous} is a finite-dimensional problem over $K\in\mathcal{K}_{\ct}\subset\mathbb{R}^{m\times n}$. Note, however, that both the objective $J_{\ct}$ and the feasible set $\mathcal{K}_{\ct}$ are nonconvex.

\subsubsection{Discrete-time LQR}
Consider a discrete-time 
LTI system
\begin{equation}\label{eq:dynamic_discrete-time}
    x_{t+1}=A x_t+B u_t,
\end{equation}
where $x_t\in\mathbb{R}^n$ is the state and $u_t\in\mathbb{R}^m$ is the input. As in the continuous-time case, let $x_0\in\mathbb{R}^n$ be a random initial state with zero mean and covariance
$
W=\mathbb{E}[x_0x_0^\tr]\in\mathbb{S}_+^n.
$ 

Given weight matrices $Q\in\mathbb{S}_+^n$ and $R\in\mathbb{S}_{++}^m$, the infinite-horizon discrete-time LQR problem is
\begin{equation}\label{eq:LQR-stochastic-noise_discrete}
\begin{aligned}
\min_{K}\quad
J_{\dt}(K):=
\mathbb{E}\!\left[
\sum_{t=0}^\infty \bigl(x_t^\tr Qx_t+u_t^\tr Ru_t\bigr)
\right]
\end{aligned}
\end{equation}
over all static state-feedback controllers
$u_t=Kx_t,$
where $K$ belongs to the set of Schur stabilizing gains
$
\mathcal{K}_{\dt}
:=
\left\{
K\in\mathbb{R}^{m\times n}
\;\middle|\;
\rho(A+BK)<1
\right\}
$ with $\rho(\cdot)$ denoting the spectral radius. 
This class is also rich enough~to~contain the discrete-time LQR optimal gain. Thus, \cref{eq:LQR-stochastic-noise_discrete} is again a finite-dimensional problem over $K\in\mathcal{K}_{\dt}\subset\mathbb{R}^{m\times n}$.~As~in~continuous time, this problem is nonconvex, and so is the set~$\mathcal{K}_{\dt}$.

\vspace{1mm}
Throughout the paper, we make a standard assumption that ensures the existence of the optimal LQR  gain 
$K^\star =K^\star_\ct$ (in continuous time) or $K^\star_\dt$ (in discrete time).
\begin{assumption} \label{assumption:stabilizable}
    $(A, B)$ is stabilizable
    and $(Q^{1/2}, A)$ is detectable.\footnote{Stabilizability and detectability in the two time domains indicate essentially the same concepts, but their mathematical characterizations slightly differ from each other; see \cite{zhou1996robust}.} 
    The covariance matrix $W$ satisfies $W\in \mathbb{S}_{+}^n\setminus\{0\}$.
\end{assumption}

Note that \Cref{assumption:stabilizable} covers the standard LQR with a deterministic initial condition $x_0$, which corresponds to $W=x_0x_0^\tr$. 
This assumption is also connected to the 
\textit{stochastic LQR} with a Gaussian process noise and $\mathcal{H}_2$ control \cite{zhou1996robust,skelton1997unified}.

\subsection{A unified reformulation of LQR}

It is known that, after a suitable reformulation using Lyapunov equations, both continuous-time and discrete-time LQRs admit an operator-based common form. 
We begin with the standard Lyapunov representation of the LQR cost.
For continuous time, if $K$ belongs to $\mathcal{K}_{\ct}$
(i.e., stabilizing), then
\[
J_{\ct}(K)=\langle Q+K^\tr RK,X\rangle,
\]
where  $\langle S,T\rangle:=\Tr(S^\tr T)$ denotes the inner product of two matrices, and
$X$ uniquely solves 
the Lyapunov equation:
\begin{equation}\label{eq:Lyapunov_eq_continuous}
    A_KX+XA_K^\tr+W=0,
\end{equation}
with $A_K:=A+BK$ being the closed-loop matrix.
The matrix $X$ can be interpreted as the cumulative covariance
$X=\int_0^\infty x(t)x(t)^\tr dt$.
Similarly, for discrete time, if $K$ belongs to $\mathcal{K}_{\dt}$ (i.e., stabilizing), then
\[
J_{\dt}(K)=\langle Q+K^\tr RK,X\rangle,
\]
where $X$ uniquely solves
the discrete-time Lyapunov equation:
\begin{equation}\label{eq:Lyapunov_eq_discrete}
    A_KXA_K^\tr-X+W=0.
\end{equation}
In this case, the matrix $X$ corresponds to
$X=\sum_{t=0}^\infty x_tx_t^\tr $.
These representations are standard; see, e.g.,
\cite{skelton1997unified,fazel2018global,watanabe2025revisiting,watanabe2026gradient}.

The two expressions already share the same objective $\langle Q+K^\tr RK,X\rangle$.
Their difference lies only in the Lyapunov equations \cref{eq:Lyapunov_eq_continuous,eq:Lyapunov_eq_discrete}: in continuous time, the system's evolution is captured by the derivative 
\begin{equation*}
    \int_0^\infty \frac{d}{dt}(x(t)x(t)^\tr)dt
=
\int_0^\infty (\dot x(t)x(t)^\tr+x(t)\dot{x}(t)^\tr )dt,
\end{equation*}leading to $A_KX+XA_K^\tr$, whereas in discrete time we use a one-step difference $\sum_{t=0}^\infty(x_{t+1}x_{t+1}-x_tx_t)$, leading to $A_KXA_K^\tr-X$.
This purely algebraic distinction motivates the following common form:
\begin{equation}\label{eq:LQR_unif_Lyapunov}
    \Psi\!\left(
   \begin{bmatrix}
       A_K\\
       I
   \end{bmatrix}
   X
   \begin{bmatrix}
       A_K\\
       I
   \end{bmatrix}^\tr
   \right)+W=0,
\end{equation}
where $\Psi:\mathbb{S}^{2n}\to\mathbb{S}^n$ is the linear operator defined by
$\Psi=\Psi_{\ct}$ in continuous time and $\Psi=\Psi_{\dt}$ in discrete time, with
\begin{align*}
    \Psi_{\ct}
    \left(
    \begin{bmatrix}
        F & G\\
        G^\tr & H
    \end{bmatrix}
    \right)
    =G+G^\tr, 
    \Psi_{\dt}
    \left(
    \begin{bmatrix}
        F & G\\
        G^\tr & H
    \end{bmatrix}
    \right)
    &=F-H.
\end{align*}
Then \cref{eq:Lyapunov_eq_continuous} is equivalent to
\cref{eq:LQR_unif_Lyapunov} with $\Psi=\Psi_{\ct}$, while
\cref{eq:Lyapunov_eq_discrete}~is equivalent to
\cref{eq:LQR_unif_Lyapunov} with $\Psi=\Psi_{\dt}$.
This operator-based~unification was used in our recent work \cite{watanabe2026gradient}, and related ideas also appear in earlier LMI formulations, such as \cite{iwasaki2005generalized,scherer2006lmi,you2016direct}.

With \cref{eq:LQR_unif_Lyapunov}, we introduce the following reformulation:
\begin{equation}\label{eq:primal_KX}
   \begin{aligned}
   p^\star=
   \min_{K,X}\quad
    & \langle Q+K^\tr RK,X\rangle\\
    \text{subject to}\quad
    & \cref{eq:LQR_unif_Lyapunov},\quad X\succeq0,
\end{aligned}
\tag{Primal}
\end{equation}
which serves as the primal LQR problem in this paper.
For each stabilizing gain $K$, the Lyapunov equations \cref{eq:Lyapunov_eq_continuous,eq:Lyapunov_eq_discrete} admit a unique solution $X\succeq0$.
In \cref{eq:primal_KX}, however, we treat $X$ as an explicit decision variable and remove the stabilizing constraint, which is convenient for deriving the dual problem. 

\begin{remark}\label{remark:equivalance}
When $W\nsucc0$, the feasible set of \cref{eq:primal_KX} may be larger than the set of stabilizing gains, since some non-stabilizing gains may also admit a feasible $X\succeq0$ satisfying \cref{eq:LQR_unif_Lyapunov}.
Thus, \cref{eq:primal_KX} is generally not an exact reformulation at the level of feasible points.
Nevertheless, \cref{eq:primal_KX} still attains its optimum at the stabilizing gain $K^\star$.
Consequently, it yields the same optimal value and recovers the same optimal controller as
\cref{eq:LQR-stochastic-noise_continuous} and \cref{eq:LQR-stochastic-noise_discrete} (as in \cref{proposition:optimal-solutions}).
For this reason, \cref{eq:primal_KX} can be viewed as a valid reformulation.
\hfill $\square$
\end{remark}

\subsection{Problem statement}

As expected, the reformulated problem \cref{eq:primal_KX} remains nonconvex, since the constraint \cref{eq:LQR_unif_Lyapunov} still couples $K$ and $X$ nonlinearly. However, this formulation \cref{eq:primal_KX} satisfies strong duality under mild conditions. In continuous time, the corresponding strong duality and its connection with Riccati equations are closely related to existing results\cite{balakrishnan2003semidefinite,bamieh2024linear,watanabe2025revisiting}. The discrete-time case has been studied in \cite{lee2018primal} under the assumption of $W\succ 0$.

The main purpose of this paper is to present a transparent and common proof strategy for both continuous-time and discrete-time LQRs through the operator $\Psi$.
In particular, our goal is to show that the two time domains share the same primal-dual mechanism, the associated algebraic Riccati equations (AREs), and the optimal gains in a common way.
This motivates the following problem.

\begin{problem}\label{Problem-1}
Under \cref{assumption:stabilizable}, derive the Lagrange dual problem of \cref{eq:primal_KX}, establish strong duality between the primal and dual problems, and characterize their solutions in a unified manner for continuous-time and discrete-time systems.
\end{problem}

%% file: arXiv-sections/section_iii.tex
\section{Strong duality in LQR}\label{section:strong-duality}

We here present the strong duality for the LQR in~\cref{eq:primal_KX}, and postpone the main proof techniques~to~\Cref{section:proof_duality}.

\subsection{Lagrange dual formulation}

We begin with deriving the Lagrange dual of \cref{eq:primal_KX}.
This follows the standard process, but we need to use the adjoint operator of $\Psi$ and its property to derive an explicit dual. 

Define the Lagrangian with the Lagrange multiplier $P\in\mathbb{S}^n$:
 \begin{align*}
        L(K,X,P)
        = 
        &
        \langle Q+K^\tr RK,X\rangle\\
        &+ \left\langle
        P,\Psi\left(\begin{bmatrix}
       A_K\\
       I
   \end{bmatrix}
   {X}
   \begin{bmatrix}
       A_K\\
       I
   \end{bmatrix}^\tr\right) +W
        \right\rangle.
    \end{align*}
Then, the dual problem is given as an abstract form below
\begin{equation} \label{eq:abstract-dual}
    \sup_{P\in\mathbb{S}^n}\quad 
    d(P)\quad
    \text{with}\quad
    d(P)=
    \inf_{K,X\succeq0} L(K,X,P).
\end{equation}

We 
derive an explicit form for function $d$.
First, we have
\begin{equation} \label{eq:two-identities}
\begin{aligned}
    \langle Q+K^\tr RK,X\rangle &= \left \langle \begin{bmatrix}
        Q & 0\\
        0& R
    \end{bmatrix},
    \begin{bmatrix}
        I\\
        K
    \end{bmatrix}X \begin{bmatrix}
        I\\
        K
    \end{bmatrix}^\tr\right \rangle, \\
\begin{bmatrix}
       A_K\\
       I
   \end{bmatrix}
   X
   \begin{bmatrix}
       A_K\\
       I
   \end{bmatrix}^\tr \!&= \!\begin{bmatrix}
       A&B\\
       I&0
   \end{bmatrix}
   \begin{bmatrix}
       I\\
       K
   \end{bmatrix}
   {X}
   \begin{bmatrix}
       I\\
       K
   \end{bmatrix}^\tr
   \begin{bmatrix}
       A&B\\
       I&0
   \end{bmatrix}^\tr\!\!.
   \end{aligned}
\end{equation}
Now, we introduce the adjoint operator $\Psi^*
$ of $\Psi$ (i.e., 
$\langle\Psi(Z),P\rangle = \langle Z,\Psi^*(P)\rangle$):
\begin{equation}\label{eq:adjoint_Psi}
            \Psi_\ct^*(P) =        \begin{bmatrix}
                0 & 1\\
                1 & 0
            \end{bmatrix}\otimes P,\quad
            \Psi_\dt^*(P)
            =
             \begin{bmatrix}
                1 & 0\\
                0 & -1
            \end{bmatrix}\otimes P.
        \end{equation}
    The derivation of the adjoint $\Psi^*
$ is standard, and we provide some details in \cref{appendix:Psi*}.
With \cref{eq:adjoint_Psi,eq:two-identities}, we 
can rewrite the Lagrangian as
\begin{align*}
        &L(K,X,P)
    \overset{(a)}{=}
    \left\langle 
    \begin{bmatrix}
        Q & 0\\
        0& R
    \end{bmatrix},
    \begin{bmatrix}
        I\\
        K
    \end{bmatrix}X \begin{bmatrix}
        I\\
        K
    \end{bmatrix}^\tr
    \right\rangle
    \nonumber
    \\
    &+ 
    \left\langle
     \begin{bmatrix}
       A & B\\
       I &0
   \end{bmatrix}^\tr {\Psi^*}(P) \begin{bmatrix}
       A & B\\
       I &0
   \end{bmatrix},
   \begin{bmatrix}
        I\\
        K
    \end{bmatrix}X \begin{bmatrix}
        I\\
        K
    \end{bmatrix}^\tr
    \right\rangle
    +
    \left\langle P,W \right\rangle,
    \end{align*}
    where (a) uses the definition of the adjoint operator $\Psi^*$ and inner product $\langle\cdot,\cdot\rangle$.
    Thus,
    a rearrangement 
    yields
\begin{equation}\label{eq:Lagrangian_transform}
  L(K,X,P)=      \langle 
    P,W
    \rangle+ 
    \left\langle 
    U(P),\begin{bmatrix}
        I\\
        K
    \end{bmatrix}X \begin{bmatrix}
        I\\
        K
    \end{bmatrix}^\tr
    \right\rangle,
\end{equation}
where we define
\begin{equation}\label{eq:Riccati-inequality}
U(P):=
     \begin{bmatrix}
       A & B\\
       I &0
   \end{bmatrix}^\tr \Psi^*(P) \begin{bmatrix}
       A & B\\
       I &0
   \end{bmatrix}
   +
   \begin{bmatrix}
        Q & 0\\
        0& R
    \end{bmatrix}. 
\end{equation}

This matrix $U(P)$ plays an important role in our subsequent discussions. Using \cref{eq:adjoint_Psi}, we can write a more explicit form as 
\begin{align}\label{eq:Riccati-inequality_explicit}
    U(P)
    =\begin{cases}
         \begin{bmatrix}
                A^\tr P+PA+Q & PB\\
                B^\tr P& R
            \end{bmatrix} &\text{(CT)}\\
            \begin{bmatrix}
                A^\tr PA-P+Q & A^\tr PB\\
                B^\tr PA& R+B^\tr PB 
            \end{bmatrix},
            &\text{(DT)}
    \end{cases}
\end{align}
where CT and DT stand for continuous time and discrete time, respectively.
From \cref{eq:Lagrangian_transform}, we have the following result.
\begin{lemma}\label{lemma:dual}
With \cref{assumption:stabilizable}, we have
\begin{equation*}
    d(P) = 
    \begin{cases}
        \langle P,W\rangle,&
        \text{if }
   U(P) \succeq 0\\
    -\infty, & \text{otherwise}.
    \end{cases}
\end{equation*}
\end{lemma}
\vspace{1mm}
\begin{proof}
If $U(P)\succeq0$,  
it is clear from \cref{eq:Lagrangian_transform} that 
    $d(P)=\min_{K,X\succeq0}L(K,X,P) =\langle P,W\rangle$, where the minimum is achieved by $X=0$.
{Otherwise, we can show $d(P)=-\infty$ by a careful eigenvalue argument.}
We present some details in \cref{appendix:proof-dual}.
\end{proof}

Consequently, the dual problem of \cref{eq:primal_KX} is expressed as
\begin{equation}
    \label{eq:dual_P}
     d^\star= \sup_{P\in\mathbb{S}^n}\quad \langle P,W\rangle\quad
        \text{subject to}\quad
   U(P)
    \succeq0.
    \tag{Dual}
\end{equation}
Note that this dual problem \cref{eq:dual_P} is an SDP. We always have $p^\star\geq d^\star$ following the standard weak duality.
\begin{remark}
The matrix $U(P)$ plays an important role for LQR in both theoretical and practical ways.
The inequality $U(P)\succeq0$ is called an \textit{algebraic Riccati inequality} (ARI) and has many classical results \cite{lancaster1995algebraic}.
On the other hand, this matrix is also closely related to the model-free \textit{Q-learning} method \cite{lewis2012reinforcement} for discrete-time systems.
In particular, 
the Q-function for the discrete-time LQR
is given by 
\[
Q(x,u) = \begin{bmatrix}
    x\\
    u
\end{bmatrix}^\tr U(P)\begin{bmatrix}
    x\\
    u
\end{bmatrix}
\]
with  
$P$ from the ARE.\hfill$\square$
\end{remark}
\subsection{Strong duality and primal and dual solutions}
Despite the nonconvexity of the LQR problems, we can guarantee strong duality between \cref{eq:primal_KX} and \cref{eq:dual_P}.
\vspace{1pt}

\begin{theorem}[Strong duality]\label{theorem:strong_duality}
    With \cref{assumption:stabilizable}, the duality gap between \cref{eq:primal_KX} and \cref{eq:dual_P} is zero, i.e., $p^\star = d^\star$.
\end{theorem}

\vspace{1pt}

The strong duality only requires the basic stabilizability and detectability assumptions and positive semidefinite $W=\mathbb{E}[x_0x_0^\tr]$ for both time domains,
while gradient dominance
\cite{fazel2018global,mohammadi2019global} requires a much stronger assumption on $W$ and has a different statement depending on the time domain \cite{watanabe2026gradient}.
Our recent work~\cite{watanabe2025revisiting} focused on the continuous-time LQR, and earlier discussions can be found in \cite{balakrishnan2003semidefinite,lee2018primal,bamieh2024linear,zhao2023global,yao2001primal}. In contrast, \Cref{theorem:strong_duality} applies to both continuous-time and discrete-time LQRs, since \cref{eq:primal_KX} and \cref{eq:dual_P} encompass both time domains via the operator $\Psi$. Moreover, our proof
provides a unified treatment by exploiting the property of $\Psi$.

We present a proof sketch for \cref{theorem:strong_duality}; full details are in \cref{section:proof_duality}. We first introduce a
rank-constrained SDP by introducing a new matrix variable, and dropping the rank constraint yields an SDP relaxation of \cref{eq:primal_KX}.
This procedure is similar to Shor's SDP relaxation for QCQPs \cite{shor1987quadratic}.
We then show that
the SDP relaxation
not only shares the same dual in \cref{eq:dual_P} but also 
is tight, i.e.,
it preserves the same optimal value as \cref{eq:primal_KX}.
This is ensured by checking strong duality and solving the KKT equation for the SDPs. One notable feature is that the KKT solution for the primal and dual SDPs must satisfy the original rank constraint, which ensures \cref{theorem:strong_duality}.

\vspace{1pt}

In addition to the strong duality, we can further characterize the optimal solutions to primal and dual LQRs.
\begin{proposition} \label{proposition:optimal-solutions}
Consider the primal-dual problems \cref{eq:primal_KX,eq:dual_P}. Suppose  \cref{assumption:stabilizable} holds. We partition the matrix $U(P)$ in \cref{eq:Riccati-inequality} as
\begin{equation}
\label{eq:ARI-U}
U(P) =  \begin{bmatrix}
        U_{11}(P)& U_{12}(P)\\
        U_{12}^\tr(P) &U_{22}(P)
    \end{bmatrix}.
\end{equation}
Then, the following statements hold. 
    \begin{enumerate}
    \item An optimal solution to
    \cref{eq:dual_P} is given by the unique stabilizing solution $P^\star \succeq 0$ to the ARE:
    \begin{equation}\label{eq:ARE_U}
    U_{11}(P) - U_{12}(P)U_{22}^{-1}(P)U_{12}^\tr(P)
    = 0.
\end{equation}
        \item An optimal solution $(K^\star,X^\star)$ to \cref{eq:primal_KX} is given by
\begin{equation}\label{eq:Kstar_U}
    K^\star = - U_{22}^{-1}(P^\star)U^{\tr}_{12}(P^\star),
\end{equation}
and $X^\star$ that solves the Lyapunov equation \cref{eq:LQR_unif_Lyapunov} for $K^\star$. 
    \end{enumerate}
    Moreover, when $W\succ0$, the solutions above are the only ones. 
\end{proposition}

\vspace{2pt}

This proposition reinterprets classical results on continuous-time and discrete-time LQRs \cite{zhou1996robust}. After substituting \cref{eq:Riccati-inequality_explicit} into \cref{eq:ARI-U}, the ARE \cref{eq:ARE_U} is indeed in the familiar form 
\begin{equation}\label{eq:ARE_duality}
      \!\!\! \!\!\!  \begin{cases}
            A^\tr P+PA+Q- PBR^{-1}B^\tr P=0
            &\text{(CT)}
            \\
            A^\tr PA-P +Q\\
            \quad-A^\tr PB(R+B^\tr PB)^{-1} B^\tr PA=0,
            &\text{(DT)}
        \end{cases}
    \end{equation}
and the optimal LQR gain \cref{eq:Kstar_U} becomes 
\begin{equation}\label{eq:Kstar_duality}
            K^\star =
            \begin{cases}
                - R^{-1}B^\tr P^\star\in\mathcal{K}_\ct & \text{(CT)}\\
                -(R+B^\tr P^\star B)^{-1}B^\tr P^\star A\in \mathcal{K}_\dt. &
                \text{(DT)}
            \end{cases}
        \end{equation}
        
In other words, the usual Riccati equation \cref{eq:ARE_duality} arises from the Schur complement of the block matrix $U(P)$. We can obtain a pair of primal optimal solutions $(K^\star,X^\star)$ from the dual optimal solution $P^\star$.    
One relatively less emphasized point is that the solution to the primal LQR problem \cref{eq:primal_KX} may not be unique, unless $W \succ 0$. Explicit continuous-time and discrete-time LQR instances can be found in \cite[Examples 3 and 4]{watanabe2026gradient}. We provide a proof of \Cref{proposition:optimal-solutions} in \cref{section:proof_duality}. 

%% file: arXiv-sections/section_iv.tex
\section{Proof of strong duality via rank-constrained SDP Relaxations}\label{section:proof_duality}

Here, we establish \cref{theorem:strong_duality,proposition:optimal-solutions} by introducing
a rank-constrained SDP for \cref{eq:primal_KX} together with careful primal and dual analysis.
\subsection{Rank-constrained SDP and its Shor's relaxation}
\label{subsection:proof-main-theorem}

We here derive a rank-constrained SDP form for \cref{eq:primal_KX}. Dropping the rank constraint directly leads to a standard SDP, which turns out to be lossless.
This process is analogous to the well-known Shor's relaxation procedure for QCQPs. We illustrate this with a simple example.  
\begin{example}\label{example:QCQP}
Consider a homogeneous QCQP:
\begin{align*}
    \min_{s\in \mathbb{R}^n} \quad & s^\tr A_0s 
    \\
    \text{subject to}\quad&
    s^\tr A_is 
    +c_i\leq 0,\, i= 1, \ldots, p, 
\end{align*}
where $A_0,A_1\ldots,A_{p}\in\mathbb{S}^n$. 
This QCQP problem is generally nonconvex. 
Now, let us introduce a new matrix variable $S=ss^\tr \in \mathbb{S}^n$, which 
is equivalent to $\operatorname{rank}(S)=1$ and $S\succeq 0$.
Thus, the QCQP problem can be equivalently rewritten as
\begin{align*}
\begin{aligned}
    \min_{S\in \mathbb{S}^n} \quad & 
    \langle A_0,S\rangle
    \\
    \text{subject to}\quad&
    \langle A_i,S\rangle
    +c_i\leq 0,\,i = 1, \ldots, p, \\
    \quad& S \succeq 0,\quad \operatorname{rank}(S)=1.
    \end{aligned}
\end{align*}
Since the nonconvexity is captured by $\operatorname{rank}(S)=1$,
dropping this rank constraint leads to a standard SDP. This process is also known as Shor's relaxation.
\hfill$\square$
\end{example}

We notice that \cref{eq:primal_KX} can be viewed as a QCQP  with 
a quadratic coupling in the matrix variable 
\begin{equation} \label{eq:matrix-varaible-V}
V=
    \begin{bmatrix}
        I\\
        K
    \end{bmatrix}X^{1/2} \in \mathbb{R}^{(n + m)\times n}.
\end{equation}
Indeed, we have $\langle Q+K^\tr RK,X\rangle
    =
    \langle Q,X\rangle+\langle R,KXK^\tr \rangle
    =
    \left\langle 
       \operatorname{diag}(Q,R),VV^\tr
    \right\rangle$ for the cost, and 
\begin{align*}
  \begin{bmatrix}
       A+BK\\
       I
   \end{bmatrix}
   X
   \begin{bmatrix}
       A+BK\\
       I
   \end{bmatrix}^\tr
   = \begin{bmatrix}
       A & B\\
       I &0
   \end{bmatrix}VV^\tr
    \begin{bmatrix}
       A & B\\
       I &0
   \end{bmatrix}^\tr.
\end{align*}
This implies that 
the LQR problem in \cref{eq:primal_KX} may be viewed as
a \textit{Matrix Quadratic Programming} (MQP) \cite{beck2012new}, which involves a quadratic cost and constraints in matrices. 

Similar to \Cref{example:QCQP}, we introduce a new variable 
$Z=VV^\tr \in \mathbb{S}^{n + m}_+$
from the matrix variable $V$ in \cref{eq:matrix-varaible-V} with $(K,X)\in \mathbb{R}^{m\times n}\times \mathbb{S}_{+}^n$. By definition, we have 
$Z\in \mathbb{S}_{+}^{m+n}$  and $\operatorname{rank}(Z)\leq n$.
We can thus relax \cref{eq:primal_KX} into the following rank-constrained SDP:
\begin{equation}\label{eq:primal_rank}
   \begin{aligned}
   \min_{Z\succeq 0}
    &\quad 
    \left\langle
    \operatorname{diag}(Q,R),Z
    \right\rangle 
    \\
    \text{subject to}&\quad
    \Psi
   \left(
   \begin{bmatrix}
       A&B\\
       I&0
   \end{bmatrix}
   Z
   \begin{bmatrix}
       A&B\\
       I&0
   \end{bmatrix}^\tr
   \right)
    +W = 0, \\
    &\quad
    \operatorname{rank}(Z)\leq n.
\end{aligned}
\end{equation}
This matrix $Z$ may be interpreted as the following covariance matrix \cite{skelton1997unified}:
\begin{equation*}
    Z = \mathbb{E}\left[\int_0^\infty \begin{bmatrix}
        x(t)\\
        u(t)
    \end{bmatrix}
    \begin{bmatrix}
        x(t)\\
        u(t)
    \end{bmatrix}^\tr
    dt\right]
    \quad\text{or}\quad
    \mathbb{E}\left[\sum_{t=0}^\infty \begin{bmatrix}
        x_t\\
        u_t
    \end{bmatrix}
    \begin{bmatrix}
        x_t\\
        u_t
    \end{bmatrix}\right].
\end{equation*}

Unless $W\succ 0$, \cref{eq:primal_rank} is generally 
not an exact reformulation of \cref{eq:primal_KX}
since the matrix $V$ in \cref{eq:matrix-varaible-V} has a further structure.\footnote{
When $W\succ0$, the feasible regions of \cref{eq:primal_KX} and \cref{eq:primal_rank} become equivalent;
see \cref{appendix:proof-equivalence} for the details.
}
As $\Psi$ is a linear operator,
the nonconvexity of \cref{eq:primal_rank} solely arises from $\operatorname{rank}(Z)\leq n$.
Hence, 
dropping the rank constraint gives a convex SDP relaxation for the primal LQR \cref{eq:primal_KX}: 
\begin{equation}\label{eq:primal_sdp}
     \begin{aligned}
   p_\texttt{sdp}^\star=\min_{Z\succeq 0}
    &\quad 
    \left\langle
    \operatorname{diag}(Q,R),Z
    \right\rangle
    \\
    \text{subject to}&\quad
    \Psi
   \left(
   \begin{bmatrix}
       A&B\\
       I&0
   \end{bmatrix}
   Z
   \begin{bmatrix}
       A&B\\
       I&0
   \end{bmatrix}^\tr
   \right)
    +W = 0,
\end{aligned}
\end{equation}

Interestingly, the Lagrange dual of this SDP \cref{eq:primal_sdp} is the same as \cref{eq:dual_P}.
We have the following result. 
\begin{proposition}\label{proposition:dual-sdp}
Consider the primal and dual LQR problems \cref{eq:primal_KX,eq:dual_P}, and the SDP relaxation \cref{eq:primal_sdp}.
Then, 1) \cref{eq:dual_P} is a Lagrange dual problem of \cref{eq:primal_sdp}; 2)
their optimal values satisfy 
$
    d^\star\leq p_\texttt{sdp}^\star \leq p^\star.
$
\end{proposition}
\begin{proof}
The first statement follows from a straightforward computation of the Lagrangian of \Cref{eq:primal_sdp}. As for the second statement, $d^\star\leq p_\texttt{sdp}^\star$ follows from standard weak duality in SDPs, and $p_\texttt{sdp}^\star \leq p^\star$ is valid from the rank relaxation. 
\end{proof}

\subsection{Proofs of \cref{theorem:strong_duality,proposition:optimal-solutions}}

It turns out that the SDP relaxation \cref{eq:primal_sdp} and the dual LQR \cref{eq:dual_P} also satisfy strong duality, and we can further explicitly characterize their primal-dual optimal solutions. 
\begin{lemma}\label{lemma:strong-duality-SDP}
Consider the SDP relaxation \cref{eq:primal_sdp} and the dual LQR \cref{eq:dual_P}. Suppose \cref{assumption:stabilizable} holds.
Then, we have: 
\begin{enumerate}
    \item The duality gap between \cref{eq:primal_sdp} and \cref{eq:dual_P} is zero, i.e., strong duality holds with 
    $p_{\sdp}^\star=d^\star$.
    \item An optimal dual solution to
    \cref{eq:dual_P} is given by the unique stabilizing solution $P^\star \succeq 0$
    to the ARE \cref{eq:ARE_U}.
    \item The optimal primal solution to \cref{eq:primal_sdp} is unique with rank at most $n$, and it admits the representation
    \begin{equation}\label{eq:Zstar}
        Z^\star =\begin{bmatrix}
            I\\
            K^\star +\Delta
        \end{bmatrix}X^\star 
        \begin{bmatrix}
            I\\
            K^\star+\Delta
        \end{bmatrix}^\tr,
    \end{equation}
    where $ K^\star$ is given in \cref{eq:Kstar_U}, $X^\star$ solves the Lyapunov equation \cref{eq:LQR_unif_Lyapunov} for $K^\star$,
    and $\Delta$ is any $m\times n$ matrix satisfying $\Delta X^\star=0$.
\end{enumerate}
When $W\succ0$, the solution $P^\star$ to \cref{eq:dual_P} is unique, and the factorization \cref{eq:Zstar} is unique with $\Delta=0$.
\end{lemma}

This result follows from the celebrated comparison theorem for (continuous- and discrete- time) AREs \cite{lancaster1995algebraic} and
the KKT condition \cite{vandenberghe1996semidefinite}.
While  \cite[Th. 1]{yao2001primal} established a similar result for $W=I$, this lemma further ensures that solution $Z^\star$ to the primal SDP \cref{eq:primal_sdp} is unique even for $W\nsucc0$.
We provide the proof details in \cref{subsection:proof-sdp-duality}.

We are now ready to prove 
\cref{theorem:strong_duality,proposition:optimal-solutions}.

\textit{Proof of \cref{theorem:strong_duality}}.
Recall that from \cref{proposition:dual-sdp} 
the optimal value of \cref{eq:primal_KX} satisfies
$p^\star_{\sdp}\leq p^\star$.
Since an optimal solution $Z^\star$ to \cref{eq:primal_sdp} is of the form \cref{eq:Zstar}, we have
\begin{equation} \label{eq:upper-bound-feasibility}
  \!  p^\star \leq 
    \langle Q+(K^\star)^\tr RK^\star ,X^\star\rangle
    \!=\!\langle \mathrm{diag}(Q,R),Z^\star\rangle
    =p^\star_{\sdp}
\end{equation}
where we use $(K^\star ,X^\star)$ as a feasible solution to \cref{eq:primal_KX}. 
Thus, \cref{eq:primal_sdp} is tight, i.e., $p^\star =  p^\star_\texttt{sdp}$. 
Combining this with $p^\star_\texttt{sdp}=d^\star$,
we conclude the strong duality between \cref{eq:primal_KX} and \cref{eq:dual_P}, i.e., 
$ p^\star = d^\star.$ \hfill$\square$

\vspace{1pt}
\textit{Proof of \cref{proposition:optimal-solutions}}. The first statement of  \cref{proposition:optimal-solutions} is the same as the second statement of \cref{lemma:strong-duality-SDP}. The second statement of \cref{proposition:optimal-solutions} essentially follows from \cref{eq:upper-bound-feasibility}, as the feasible point $(K^\star,X^\star)$ achieves the optimal value $p^\star$. This pair $(K^\star ,X^\star)$ is an optimal solution to \cref{eq:primal_KX}. When $W\succ 0$, 
the dual optimal solution is unique by \cref{lemma:strong-duality-SDP}.
Moreover, we also have $X^\star\succ0$, which implies that $\Delta$ satisfying $\Delta X^\star=0$ is $0$.
Then, the factorization of $Z^\star$ in \cref{eq:Zstar} is unique, leading to the uniqueness of $(K^\star,X^\star)$ by the construction of \cref{eq:primal_sdp}.
\hfill$\square$

\begin{remark}
For nonconvex QCQPs with a single inequality constraint,
we have the famous \textit{S-lemma}, i.e.,
the strong duality and tightness of Shor's SDP relaxation hold 
    under a mild assumption.
    Our analysis shows that \cref{eq:primal_KX} can be viewed as an MQP \cite{beck2012new}.
    One interesting
    direction is 
    to address more advanced nonconvex control problems
    from an MQP perspective (e.g., using an extension of S-lemma \cite{beck2012new}).
    \hfill
    $\square$
\end{remark}

\subsection{Proof of \cref{lemma:strong-duality-SDP}}\label{subsection:proof-sdp-duality}

We finally prove the three statements in \cref{lemma:strong-duality-SDP}. This requires classical results in continuous- and discrete-time AREs, as well as careful KKT analysis between \cref{eq:primal_sdp} and~\cref{eq:dual_P}.  

We first introduce two technical lemmas on the ARE \cref{eq:ARE_duality}. 
\begin{lemma}[Comparison theorem {\cite[Th. 9.1.1, Th. 13.1.1]{lancaster1995algebraic}}]\label{lemma:comparison}
Suppose \cref{assumption:stabilizable} holds.
Then, the stabilizing solution $P^\star$ to the ARE \cref{eq:ARE_duality} is unique and maximal, i.e., it satisfies
    $P^\star\succeq P$
for any $P\in\mathbb{S}^n$ satisfying $U(P)\succeq0$. This stabilizing solution $P^\star$ is also positive semidefinite. 
\end{lemma}

\begin{lemma}\label{lemma:dual-strict-feasibility}
Suppose \cref{assumption:stabilizable} holds.
Then, there exists a matrix $P_+\in\mathbb{S}^n$ such that
$U(P_+)\succ0$.
\end{lemma}

The continuous-time case in \cref{lemma:dual-strict-feasibility} is from \cite[Th. 2.23]{scherer1990riccati}; see \Cref{appendix:proof-dual-strict-feasibility} for details in the discrete-time case.

\textbf{Proof of statements 1) and 2).}
According to \cite[Th. 3.1]{vandenberghe1996semidefinite},
for establishing the first statement, it suffices to confirm 
$|d^\star|<\infty$ and
the strict feasibility of \cref{eq:dual_P}, i.e., the optimal value of \cref{eq:dual_P} is finite and $U(P)\succ 0$ for some $P\in\mathbb{S}^n$.
We can verify the former condition $|d^\star|<\infty$ by \Cref{lemma:comparison}, which further proves the second statement.
Since $P^\star$ is maximal and $W \succeq 0$, this solution achieves the largest value among all feasible points in \cref{eq:dual_P}. Thus, we have proved the second statement in \Cref{lemma:strong-duality-SDP},
which also implies $0\leq d^\star= \langle P^\star,W \rangle<\infty$.

Next, to show dual strict feasibility,
\Cref{lemma:dual-strict-feasibility} guarantees the existence of some $P\in\mathbb{S}^n$ satisfying $U(P)\succ0$.
This lemma indeed ensures the strict feasibility of \cref{eq:dual_P}. 
Consequently,
\cite[Th. 3.1]{vandenberghe1996semidefinite}
ensures the first statement, i.e.,
the strong duality holds.
The second statement is clear from the above argument.

\textbf{Proof of statement 3)}
By the strong duality in the first statement,
\cite[Th. 3.1]{vandenberghe1996semidefinite}
further confirms that
a minimizer $Z^\star\in\mathbb{S}_+^{n+m}$  to \cref{eq:primal_sdp} exists.
Now,
any such $Z^\star$ is characterized by
the following KKT condition between \cref{eq:primal_sdp} and \cref{eq:dual_P}:
\begin{subequations}\label{eq:KKT}
\begin{equation}
\label{eq:KKT-primal}
     \Psi
   \left(
   \begin{bmatrix}
       A&B\\
       I&0
   \end{bmatrix}
   Z^\star
   \begin{bmatrix}
       A&B\\
       I&0
   \end{bmatrix}^\tr
   \right)
    +W = 0,\,Z^\star \succeq0,
\end{equation}
\begin{minipage}{0.49\columnwidth}
\vspace{-4mm}
\begin{equation}
\label{eq:KKT-dual}
U(P^\star)\succeq0,\,P^\star\in\mathbb{S}^n,
\end{equation}
\end{minipage}
\hfill
\begin{minipage}{0.49\columnwidth}
\vspace{-4mm}
\begin{equation}
 \label{eq:KKT-slackness}
\langle Z^\star,U(P^\star)\rangle=0
,
\end{equation}
\end{minipage}
\end{subequations}

\vspace{8pt}
\noindent 
where $P^\star$ solves \cref{eq:dual_P}.
From the statement 2) of \cref{lemma:strong-duality-SDP},
the stabilizing solution $P^\star\succeq0$
to the ARE \cref{eq:ARE_U}
constitutes a solution to \cref{eq:dual_P}.
Now, recall
\cref{eq:ARE_U} for $P^\star$, i.e.,
    $U_{11}(P^\star)-U_{12}(P^\star)U_{22}^{-1}(P^\star)U_{12}^\tr (P^\star)=0$.
Substituting this into the complementary slackness \cref{eq:KKT-slackness} yields
\begin{align}
     &0=\left\langle Z^\star ,  U(P^\star)   \right\rangle=\left\langle Z^\star , \begin{bmatrix}
        U_{11}(P^\star )& U_{12}(P^\star )\\
        U_{12}^\tr(P^\star) &U_{22}(P^\star)
    \end{bmatrix}\right\rangle \label{eq:slackness-proof}
    \\
    &
    =
    \Tr \left(
    Z^\star
    \begin{bmatrix}
        U_{12}(P^\star)U_{22}^{-1/2}(P^\star)\\
        U_{22}^{1/2}(P^\star)
    \end{bmatrix}
    \begin{bmatrix}
        U_{12}(P^\star)U_{22}^{-1/2}(P^\star)\\
        U_{22}^{1/2}(P^\star)
    \end{bmatrix}^\tr
    \right) \nonumber
\end{align}
from
 the ARE \cref{eq:ARE_U} with $P=P^\star$ and the factorization
\begin{align*}
  U(P^\star) =&  \begin{bmatrix}
        U_{12}(P^\star)U_{22}^{-1}(P^\star)U_{12}^\tr(P^\star)& U_{12}(P^\star )\\
        U_{12}^\tr(P^\star) &U_{22}(P^\star)
    \end{bmatrix}\\
    =&
     \begin{bmatrix}
        U_{12}(P^\star)U_{22}^{-1/2}(P^\star)\\
        U_{22}^{1/2}(P^\star)
    \end{bmatrix}
    \begin{bmatrix}
        U_{12}(P^\star)U_{22}^{-1/2}(P^\star)\\
        U_{22}^{1/2}(P^\star)
    \end{bmatrix}^\tr.
\end{align*}

Note that \cref{eq:KKT-slackness} implies
 $   \operatorname{rank}(Z^\star)+\operatorname{rank}(U(P^\star))\leq n+m.$ 
Thus, combining this with $\operatorname{rank}(U(P^\star))=\operatorname{rank}(U_{22}(P^\star))=\operatorname{rank}(R)= m$
(from
\cref{eq:Riccati-inequality_explicit} and
$P^\star\succeq0$),
we obtain $\operatorname{rank}(Z^\star)\leq n$, which enforces $Z^\star$ to be of the~form
\begin{equation} \label{eq:Z-factorization}
    Z^\star = 
    \begin{bmatrix}
    V_1\\
    V_2
    \end{bmatrix}
    \begin{bmatrix}
    V_1\\
    V_2
    \end{bmatrix}^\tr 
\end{equation}
for $V_1\in\mathbb{R}^{n\times n}$, $V_2\in\mathbb{R}^{m\times n}$.
Substituting this into \cref{eq:slackness-proof} gives
\begin{align*}
&\langle Z^\star,U(P^\star) \rangle
=
\left\|
    \begin{bmatrix}
    V_1\\
    V_2
    \end{bmatrix}^\tr 
    \begin{bmatrix}
        U_{12}(P^\star)U_{22}^{-1/2}(P^\star)\\
        U_{22}^{1/2}(P^\star)
    \end{bmatrix}
\right\|_F^2\\
=&\|V_1^\tr U_{12}(P^\star)U_{22}^{-1/2}(P^\star)+V_2^\tr  U_{22}^{1/2}(P^\star)\|_F^2=0.
\end{align*}
Hence, we obtain $V_1^\tr U_{12}(P^\star)U_{22}^{-1/2}(P^\star)+V_2 ^\tr U_{22}^{1/2}(P^\star)=0$,
which implies
    $V_2=-\left(U_{22}(P^\star)\right)^{-1}U_{12}^\tr(P^\star)V_1
    =K^\star V_1$
from \cref{eq:Kstar_U}.  We know from \cref{eq:Z-factorization} that $Z^\star$ must be of the form
\begin{equation}\label{eq:Zstar-proof}
    Z^\star = 
    \begin{bmatrix}
    V_1\\
    K^\star V_1
    \end{bmatrix}
    \begin{bmatrix}
    V_1\\
    K^\star V_1
    \end{bmatrix}^\tr = 
    \begin{bmatrix}
    I\\
    K^\star 
    \end{bmatrix}V_1V_1^\tr 
    \begin{bmatrix}
    I\\
    K^\star
    \end{bmatrix}^\tr.
\end{equation}
The construction of \cref{eq:primal_sdp}
indicates
that
\cref{eq:KKT-primal} is reduced to the Lyapunov equation \cref{eq:LQR_unif_Lyapunov} for $K=K^\star$, and 
$V_1V_1^\tr$ is uniquely determined as
$V_1V_1^\tr=X^\star$.
Hence, the primal problem is uniquely solved
by $Z^\star$ in \cref{eq:Zstar}.

\paragraph*{Proof of the uniqueness when $W\succ0$}
First, we verify that $P^\star$ from the ARE \cref{eq:ARE_U} uniquely solves \cref{eq:dual_P}.
Indeed, if we have two different solutions as
$\langle \hat P ,W\rangle=\langle P^\star ,W\rangle$ for $\hat P\neq P^\star$,
 we have $\Tr((P^\star-\hat P)W)=0$, and thus 
 $$(P^\star-\hat P)W=0$$
 from $P^\star\succeq \hat{P}$
 (Recall that 
 $\Tr(AB)=0$ for $A,B\succeq0$ implies $AB=0$).
However, this is a contradiction because $P^\star-\hat P\neq0$.
Hence, the dual problem has the unique solution $P^\star$.

Finally, $W\succ0$ implies $X^\star\succ0$, so that
$\Delta X^\star=0$ yields $\Delta=0$. Hence, the factorization in \cref{eq:Z-factorization} is also unique.

\begin{remark}
In our proof, we identify the form of $K^\star$ in a constructive manner by using the rank structure as in \cref{eq:slackness-proof}, \cref{eq:Z-factorization}, and \cref{eq:Zstar-proof}.
Namely, we do not need the knowledge or a guess of $K^\star$ for the strong duality.
This aspect may be viewed as a benefit, compared with the earlier results \cite{lee2018primal,bamieh2024linear}.
\hfill$\square$
\end{remark}

%% file: arXiv-sections/appendix.tex
\subsection{Derivation of the adjoint operator $\Psi^*$}\label{appendix:Psi*}

We can derive the adjoint operators from a simple algebra.
Specifically, for a block matrix
\begin{equation*}
Z=
\begin{bmatrix}
F & G\\
G^\tr  & H
\end{bmatrix}
\in \mathbb{S}^{2n},
\end{equation*}
in the continuous-time case, we have
\begin{equation*}
    \Psi_{\ct}(Z)=G+G^\tr .
\end{equation*}
Then, for any $P\in\mathbb{S}^n$,
it is not difficult to see that
\begin{align*}
\inner{\Psi_{\ct}(Z)}{P}
=
\inner{G+G^\tr }{P}  
=
\left\langle
Z,
\begin{bmatrix}
0 & P\\
P & 0
\end{bmatrix}
\right\rangle.
\end{align*}
Hence,
\begin{equation*}
    \Psi_{\ct}^\ast(P)
    =
    \begin{bmatrix}
    0 & P\\
    P & 0
    \end{bmatrix}.
\end{equation*}
Similarly, for the discrete-time case,
\begin{equation*}
    \Psi_{\rm dt}(Z)=F-H,
\end{equation*}
we have
\begin{align*}
\inner{\Psi_{\dt}(Z)}{P}=
\inner{F-H}{P} 
=
\left\langle
Z,
\begin{bmatrix}
P & 0\\
0 & -P
\end{bmatrix}
\right\rangle,
\end{align*}
so that
\begin{equation*}
    \Psi_{\rm dt}^\ast(P)
    =
    \begin{bmatrix}
    P & 0\\
    0 & -P
    \end{bmatrix}.
\end{equation*}

\subsection{Proof of \cref{lemma:dual}}\label{appendix:proof-dual}
Here, we show $d(P)=-\infty$ if $U(P)\nsucceq0$.
First,
if $U(P)$ is not positive semidefinite,
there exists an eigenvector $v=[v_1^\tr,v_2^\tr ]^\tr\neq 0$ such that
   $U(P) v
    = -\lambda v$
with some $\lambda>0$.

We separately show the two cases, i.e., $v_1\neq0$ and $v_1=0$.
When $v_1\neq0$, we can choose  
    $$X=\eta v_1v_1^\tr \succeq 0,\quad K=v_2v_1^\tr$$
with $\eta>0$. 
In addition, we can set $\|v_1\|=1$ without loss of generality.
Direct computations now ensure that
\begin{equation*}
    \begin{bmatrix}
        I\\
        K
    \end{bmatrix}X\begin{bmatrix}
        I\\
        K
    \end{bmatrix}^\tr
    = \eta \begin{bmatrix}
        v_1v_1^\tr & v_1v_2^\tr\\
        v_2^\tr v_1& v_2v_2^\tr
    \end{bmatrix} = \eta vv^\tr ,
\end{equation*}
which gives
    $L(K,X,P)= \langle P,W\rangle- \eta\lambda \|v\|^2$
from \cref{eq:Lagrangian_transform}.
We now see $d(P)=-\infty$ as $\eta \to \infty$. 

On the other hand, if $v_1=0$, we have $v_2\neq0$ by $v\neq0$. In addition, we have
\begin{equation*}
    v_2^\tr U_{22}(P)v_2=-\lambda \|v_2\|^2<0
\end{equation*}
for the bottom-right block of $U(P)$.
We now choose $$X=ww^\tr \succeq 0,\quad K= \eta v_2w^\tr$$ with a unit vector $w$ and $\eta>0$. With this choice, it is not difficult to verify that
 $$L(K,X,P)= \langle P,W\rangle+ w^\tr U_{11}(P)w- \eta^2 \lambda \|v_2\|^2,$$
 where $U_{11}(P)$ is the top-left block of $U(P)$.
 Consequently, we obtain $d(P)=-\infty$ by $\eta\to\infty$.
This completes the proof.

\subsection{Equivalence between \cref{eq:primal_KX} and the rank-constrained SDP \cref{eq:primal_rank}}\label{appendix:proof-equivalence}

Here, we verify that the feasible domains of
\cref{eq:primal_KX} and \cref{eq:primal_rank} are equivalent
when $W\succ0$.
Since any feasible $(K,X)$ to \cref{eq:primal_KX} can constitute a feasible point for \cref{eq:primal_rank} by its construction,
it suffices to show the converse direction.

Consider a feasible $Z\succeq0$ to \cref{eq:primal_rank}.
Since $\mathrm{rank}(Z)\leq n$,
there exist $V_1\in\mathbb{R}^{n\times n}$ and $V_2\in\mathbb{R}^{m\times n}$ such that
$$Z = \begin{bmatrix}
    V_1\\
    V_2
\end{bmatrix}\begin{bmatrix}
    V_1\\
    V_2
\end{bmatrix}^\tr.$$
The constraint of \cref{eq:primal_rank} is then reduced to
\begin{equation*}
    \begin{cases}
        (AV_1+BV_2)V_1^\tr +V_1(AV_1+BV_2)^\tr +W=0&\text{(CT)}\\
        (AV_1+BV_2)(AV_1+BV_2)^\tr-V_1V_1^\tr +W=0 ,&\text{(DT)}
    \end{cases}
\end{equation*}
where CT and DT stand for continuous time and discrete time, respectively.

Now, it is not difficult to show $\mathrm{rank}(V_1)=n$, i.e., $V_1$ is non-singular.
Indeed, for the continuous-time case, if there exists a vector $v\neq0$ satisfying $V_1^\tr v=0$,
we have
\begin{equation*}
    v^\her ((AV_1+BV_2)V_1^\tr +V_1(AV_1+BV_2)^\tr +W)v
    = v^\her Wv=0,  
\end{equation*}
which is a contradiction from $W\succ0$.
Hence, $\mathrm{rank}(V_1)=n$.
For the discrete-time case, it is clear that
\begin{equation*}
    V_1V_1^\tr = W+(AV_1+BV_2)(AV_1+BV_2)^\tr
    \succeq W\succ0,
\end{equation*}
which also implies $\mathrm{rank}(V_1)=n$.
Hence, we can rewrite any feasible $Z$ to \cref{eq:primal_rank} as
$$Z = \begin{bmatrix}
    I\\
    V_2V_1^{-1}
\end{bmatrix}V_1V_1^\tr \begin{bmatrix}
    I\\
    V_2V_1^{-1}
\end{bmatrix}^\tr.$$
By assigning $K=V_2V_1^{-1}$ and $X=V_1V_1^\tr$, this pair $(K,X)$ constitutes a feasible point of \cref{eq:primal_KX}.
Therefore, \cref{eq:primal_KX} and \cref{eq:primal_rank} have the equivalent feasible region.

\subsection{Proof of \cref{lemma:dual-strict-feasibility}}\label{appendix:proof-dual-strict-feasibility}

Here, we prove that 
$U(P_+)\succ0$ for some $P_+\in\mathbb{S}^n$.
The continuous-time case follows from \cite[Th. 2.23]{scherer1990riccati}.
One can also verify the discrete-time case in a similar manner.
For completeness, we provide a proof below.

First, we know that \cref{assumption:stabilizable} ensures the existence of the stabilizing solution $P^\star$ to the ARE \cref{eq:ARE_U}.
Let $Q(\epsilon)=Q-\epsilon I$.
Then, \cite[Th. 2.1]{konstantinov1993perturbation}\footnote{Theorem 2.1 from \cite{konstantinov1993perturbation} shows a more general result than our purpose.
For completeness, we provide in \cref{appendix:proof-DARE} a proof of its special case that suffices to establish \cref{lemma:dual-strict-feasibility}.
}
guarantees that
for any $\epsilon\in[0,\bar\epsilon)$ with a sufficiently small $\bar\epsilon>0$, the ARE
\begin{equation}\label{eq:ARE-eps}
    A^\tr PA-P+Q(\epsilon)-
    A^\tr PB(R+B^\tr PB)^{-1}B^\tr PA=0
\end{equation}
has a solution $P=P_\epsilon$ that is symmetric from the form of \cref{eq:ARE-eps}.
Further,
$P_\epsilon$ continuously depends on $\epsilon$ in a sufficiently small neighborhood of $\epsilon=0$.
Then, since $P^\star=P_0\succeq0$ ensures
$$R+B^\tr P_0 B\succ 0,$$ 
there exist a sufficiently small $\hat{\epsilon}\in(0,\bar{\epsilon})$ and
a symmetric matrix $P_+=P_{\hat{\epsilon}}\in\mathbb{S}^n$ 
such that  we have both
$R+B^\tr P_{\hat{\epsilon}} B\succ0$
and
\cref{eq:ARE-eps}, i.e.,
\begin{equation*}
    A^\tr P_{\hat{\epsilon}} A-P_{\hat{\epsilon}} +Q(\hat \epsilon)-
    A^\tr P_{\hat{\epsilon}} B(R+B^\tr P_{\hat{\epsilon}} B)^{-1}B^\tr P_{\hat{\epsilon}} A=0.
\end{equation*}
Hence, by
\begin{equation*}
    A^\tr P_{\hat{\epsilon}} A-P_{\hat{\epsilon}} +Q-
    A^\tr P_{\hat{\epsilon}} B(R+B^\tr P_{\hat{\epsilon}} B)^{-1}B^\tr P_{\hat{\epsilon}} A=\hat{\epsilon} I\succ0,
\end{equation*}
the Schur complement yields
\begin{equation*}
    U(P_{\hat\epsilon})
    = 
    \begin{bmatrix}
         A^\tr P_{\hat{\epsilon}} A-P_{\hat{\epsilon}}+Q&A^\tr P_{\hat{\epsilon}} B \\
        B^\tr P_{\hat{\epsilon}} A &R+B^\tr P_{\hat{\epsilon}}B
    \end{bmatrix}\succ0.
\end{equation*}
We finally remark that
when $Q\succ0$, the existence of $P\in\mathbb{S}^n$ satisfying $U(P)\succ0$ is obvious, as $U(0)=\operatorname{diag}(Q,R)\succ0$. 

\subsection{Solutions to perturbed discrete-time AREs}\label{appendix:proof-DARE}

Here, we show 
that even if we add a slight perturbation to $Q$ in the ARE \cref{eq:ARE_U} in discrete time, we still have a symmetric solution with an analytic dependence on the perturbation.
This is 
a special case of \cite[Th. 2.1]{konstantinov1993perturbation}, which 
is sufficient to show \cref{lemma:dual-strict-feasibility} for discrete-time systems.
\begin{proposition}
Suppose \cref{assumption:stabilizable} holds.
Then, for any $\epsilon\in(-\bar \epsilon,\bar \epsilon)$ with a sufficiently small $\bar{\epsilon}>0$,
the ARE \cref{eq:ARE-eps} with $Q(\epsilon)=Q-\epsilon I$ has a symmetric solution $P_\epsilon$.
Moreover, $P_\epsilon$ depends on $\epsilon$ analytically over $(-\bar \epsilon,\bar\epsilon)$.
\end{proposition}
\begin{proof}
First, let us define the mapping 
$F:\mathbb{S}^n\times\mathbb{R}\to\mathbb{R}^{n^2}$
\begin{align*}
F(P,\epsilon) = \operatorname{vec}&(
A^\tr PA-P+Q(\epsilon)\\
&-
    A^\tr PB(R+B^\tr PB)^{-1}B^\tr PA).
\end{align*}
By \cref{assumption:stabilizable}, it is clear that $F(P^\star,0)=0$.
To use the implicit function theorem around $(P^\star,0)$,
we compute $\nabla_P F(P^\star,\epsilon)|_{\epsilon=0}.$
Now, it is not very difficult to see that
\begin{align*}
    &
    \operatorname{vec}^{-1}
    \left(F(P^\star+dP,\epsilon)-F(P^\star,\epsilon)\right)\\
    =&
    A^\tr dPA-dP\\
    &+ A^\tr P^\star B(R+B^\tr P^\star B)^{-1}B^\tr dPB(R+B^\tr P^\star B)^{-1}B^\tr P^\star A\\
    &-A^\tr dPB(R+B^\tr P^\star B)^{-1}B^\tr P^\star A  \\
    &-A^\tr P^\star B(R+B^\tr P^\star B)^{-1}B^\tr dPA
    +O(\|dP\|_F^2).
\end{align*}
Using $K^\star= -(R+B^\tr P^\star B)^{-1}B^\tr P^\star A$, we can write
\begin{align*}
    &\operatorname{vec}^{-1}
    \left(F(P^\star+dP,\epsilon)-F(P^\star,\epsilon)\right)\\
    =& A^\tr dPA-dP+ (K^\star)^\tr B^\tr dP BK^\star\\
    &+A^\tr dPBK^\star
    +(K^\star)^\tr B^\tr dPA+ O(\|dP\|_F^2)\\
    =&
    (A+BK^\star)^\tr dP (A+BK^\star)-dP+
    O(\|dP\|_F^2).
\end{align*}
Hence, by the standard formula $\operatorname{vec}(S_1VS_2)=(S_2^\tr \otimes S_1)\operatorname{vec}(V)$,
\begin{align*}
    &F(P^\star+dP,\epsilon)-F(P^\star,\epsilon)\\
=& \left((A_{K^\star}^\tr \otimes A_{K^\star}^\tr) -I\right)
\operatorname{vec}(dP)
+O(\|dP\|_F^2).
\end{align*}
It then follows that
\begin{equation*}
  \nabla_P F(P^\star,\epsilon)|_{\epsilon=0}  
  = (A_{K^\star}^\tr \otimes A_{K^\star}^\tr) -I,
\end{equation*}
which is invertible from $\rho(A_{K^\star})<1$.
Hence, the classical (vector-valued) implicit function theorem \cite[Th. 2.3.5]{krantz2002implicit} 
around $(P^\star,0)$ for $F(P,\epsilon)=0$
confirms that
there exists a sufficiently small $\bar \epsilon>0$ such that
for any $\epsilon\in(-\bar\epsilon,\bar\epsilon)$, we have
$$F(P_\epsilon,\epsilon)=0$$ with some $P_\epsilon$, and 
$P_\epsilon$ is an analytic function of $\epsilon$ on $(-\bar\epsilon,\bar\epsilon)$.
Finally, since the implicit function theorem is applied to the space of symmetric
matrices, the resulting local solution $P_\epsilon$ remains symmetric.
\end{proof}